\documentclass[english,11pt,reqno,twoside]{amsart}

\usepackage[T1]{fontenc}      
\usepackage[utf8]{inputenc}     
\usepackage[english]{babel} 
\usepackage{lmodern}

\usepackage{amsmath,amssymb,amsthm}  
\usepackage{mathtools}         
\usepackage{mathrsfs}           
\usepackage{braket}   

\usepackage{graphicx}         
\usepackage{xcolor}             
\usepackage[normalem]{ulem}

\usepackage[noadjust]{cite}   
\usepackage{url}                

\usepackage[left=2.1cm,right=2.1cm,top=2cm,bottom=2cm]{geometry}
\usepackage{indentfirst}       
\usepackage{microtype}    
\usepackage{enumitem}

\usepackage[colorlinks=true,urlcolor=blue,
citecolor=red,linkcolor=blue,linktocpage,pdfpagelabels,
bookmarksnumbered,bookmarksopen]{hyperref}
\usepackage[hyperpageref]{backref}
\usepackage[colorinlistoftodos]{todonotes}

\newtheorem{thm}{Theorem}[section]

\newtheorem{lem}[thm]{Lemma}
\newtheorem{prop}[thm]{Proposition}

\theoremstyle{definition}

\theoremstyle{remark}
\newtheorem{rem}[thm]{Remark}

\numberwithin{equation}{section}
\DeclareMathSymbol{\C}{\mathalpha}{AMSb}{"43}

\newcommand{\R}{{\mathbb{R}}}

\newcommand{\X}{\mathcal{X}}

\newcommand{\inte}{\int_{\mathbb{R}^N}}

\begin{document}
	\title[the Pseudo-Relativistic Schr\"{o}dinger Equation with Logarithmic Nonlinearity]{Existence and Limiting Profiles of Normalized Travelling Wave   Solutions for the Pseudo-Relativistic Schr\"{o}dinger Equation with Logarithmic Nonlinearity}
	
	\author[P. d'Avenia]{Pietro d'Avenia}
	\address{P. d'Avenia
		\newline\indent Dipartimento di Meccanica, Matematica e Management,\newline\indent
		Politecnico di Bari
		\newline\indent
		Via Orabona 4,  70125  Bari, Italy}
	\email{pietro.davenia@poliba.it}
	
	\author[Q. He]{Qihan  He}
	\address{Q. He
		\newline \indent School of Mathematics \& Center for Applied Mathematics of Guangxi (Guangxi University)
		\newline \indent Guangxi University, 
		\newline \indent Nanning, Guangxi, P. R. China}
	\email{heqihan277@gxu.edu.cn}
	
	\author[A. Pomponio]{Alessio Pomponio}
	\address{A. Pomponio
		\newline\indent Dipartimento di Meccanica, Matematica e Management,\newline \indent
		Politecnico di Bari
		\newline\indent
		Via Orabona 4,  70125  Bari, Italy}
	\email{alessio.pomponio@poliba.it}

	\author[L. Yang]{Lianfeng Yang}
	\address{L. Yang
		\newline \indent School of Mathematics and Statistics, 
		\newline \indent Beijing Institute of Technology, 
		\newline \indent Haidian, 100081 Beijing, China}
	\email{yanglianfeng2021@163.com}
	
	\date{}

	

	\begin{abstract}
		We study the existence and asymptotic behaviour of normalized solutions to the following pseudo-relativistic Schr\"{o}dinger equation with logarithmic nonlinearity
		\begin{equation*}
			(\sqrt{-\Delta+m^2 }-m )u+i(v\cdot \nabla )u+\lambda u = u\log|u|^2+|u|^{p-2}u, \qquad \text{in }\R^N,
		\end{equation*}
		under the mass constraint
		$$\|u\|_2^2=a,$$
		where $m,a>0$, $2<p\le\frac{2N}{N-1}$ with $N\ge 2$, $v\in \mathbb{R}^N$ is the travelling velocity  with $|v|<1$, and $\lambda\in\R$ appears as Lagrange multiplier, as minima of the corresponding energy on the constraint.
		
		By applying variational method, we first provide a complete classification of the existence and nonexistence of such minima.
		In particular, for the mass-critical case $p=2+\frac{2}{N}$, we show that there exists a constant $a^\ast_v$ which is a threshold for the existence. 
		Based on this, we analyse the blow-up behaviour of such minimizers as $a$ approaches $a^\ast_v$ from below. Finally, we investigate the limiting profiles of  minimizers to problem  when $\lim\limits_{n\to\infty}a_n=a_0\in(0,+\infty)$ with $\{a_n\}\subset(0,+\infty)$ in the mass-subcritical case $2<p<2+\frac{2}{N}$ and $\lim\limits_{n\to\infty}a_n= a_0\in(0,a^\ast_v)$ with $\{a_n\}\subset(0,a^\ast_v)$ in the mass-critical case $p=2+\frac{2}{N}$, respectively.     
	\end{abstract}

	\maketitle

	\section{Introduction}
	This paper is concerned with the following 
	pseudo-relativistic Schr\"{o}dinger equation with logarithmic nonlinearity
	\begin{equation*}
		i\partial _t \psi=(\sqrt{-\Delta+m^2}-m )\psi- \psi\log|\psi|^2 - |\psi|^{p-2}\psi,~(t,x)\in\mathbb{R}^+\times\mathbb{R}^N,
	\end{equation*}
	where  the parameters $N\geq2$, $2<p\le\frac{2N}{N-1}$, and $\psi=\psi(t,x)$ is a complex-valued wave function, and the fractional operator $\sqrt{-\Delta+m^2}-m$ is the kinetic energy operator of a relativistic particle with mass $m>0$. The fractional operator $\sqrt{-\Delta+m^2}-m$ is also usually called the pseudo-differential operator and is defined by its symbol $\sqrt{{|k|}^2+m^2}-m$, namely,
	$$\left(\sqrt{-\Delta+m^2}-m\right)\psi=\mathcal{F}^{-1}\left [\left(\sqrt{{|k |}^2+m^2}-m\right)\mathcal{F}\left[\psi\right]\right ],$$
	where $\mathcal{F}$ and $\mathcal{F}^{-1}$ denote the Fourier transform and the inverse Fourier transform, respectively. Moreover, it is widely known that there is also a deep connection between $\sqrt{-\Delta+m^2}-m$ and the theory of stochastic processes. More precisely, 
	the operator $-\left(\sqrt{-\Delta+m^2}-m\right)$ is the infinitesimal generator of a L\'evy process
	$(X_t^{1,m})_{t\ge0}$, called the {\em relativistic $1$-stable process}, whose characteristic function is
	\[
	\mathbb{E}^{0}[e^{ik\cdot X_t^{1,m}}]
	=
	e^{-t(\sqrt{|k|^{2}+m^{2}}-m)},
	\qquad k\in\mathbb{R}^{N}.
	\]
	We refer to 
	\cite{CMS,BMR} and references therein for a more detailed discussion.

	In the past three decades, the structure and dynamic behaviour of the standing waves of the form $\psi(t,x)=e^{i\lambda t}\varphi(x)$ to equation 
	\begin{equation}\label{mu0}
		i\partial _t \psi=(\sqrt{-\Delta+m^2}-m )\psi - |\psi|^{p-2}\psi,~(t,x)\in\mathbb{R}^+\times\mathbb{R}^N,
	\end{equation}
	have been extensively studied.
	For instance, \cite{A1,A2,CS} investigated the existence and some qualitative properties such as symmetry or sign definiteness, and non-relativistic limit of the ground state solutions for such a problem.
	
	Actually,  a substantial body of literature has been devoted to the {\em pseudo-relativistic Hartree equation}, namely equation \eqref{mu0}  in which the power nonlinearity is replaced by $(|x|^{-1}\ast|\psi|^{2}) \psi$ (see \cite{CZN,FL,GZ,GZ1,HL,L1,L2,L3,F,LY,HYZ25} and references therein).

	On the other hand in the seminal paper \cite{C83}, the logarithmic Schr\"{o}dinger equation
	\begin{equation}\label{log}
		i\partial _t \psi=-\Delta\psi-\psi\log|\psi|^2, \qquad\text{in }\R^N,
	\end{equation}
	is studied. Looking for standing waves, due to the singularity of the logarithm at the origin the energy functional associated to \eqref{log} fails to be finite as well of class $C^1$ on $H^{1}(\R^N)$. Thus, the classical critical point theory cannot be directly applied. 
	In order to overcome these difficulties, in \cite{C83} a suitable Banach space has been introduced. Later on, this problem has been considered using different methods: non-smooth critical point theory in \cite{DMS,MA}, direction derivative and constrained minimization method in \cite{shuai},  penalization technique in \cite{KZ}.
	Subsequently, \cite{A17,DSZ} extended the previous results to the fractional case.

	Differently from the above works on standing waves, the authors of \cite{FJL07} investigated the existence, nonexistence and orbital stability of travelling solitary waves, namely, searching for solution of the form
	\begin{equation*}
		\psi(t,x)=e^{i\lambda t}\varphi\left(x-vt\right),
	\end{equation*}
	for the pseudo-relativistic Hartree equation. By applying variational method, they found that, for speeds $|v|<1$, there exists a threshold $N_c(v)>0$ so that the constrained minimization problem has no minimizers if the mass constraint $\|\varphi\|_2^2=\mathcal{N}\geq N_c(v)$ and admits at least one minimizer if $0<\mathcal{N}<N_c(v)$, where they mainly consider the influence of mass on the existence of solutions.
	
	Inspired by \cite{FJL07} and in the wake of \cite{DPSL26,HYZ25}, motivated by the fact that physicists are often interested in studying solutions having prescribed $L^2$-norm (physically, such solutions are called the normalized solutions), in this paper, we focus on studying the following stationary problem
	\begin{equation}\label{eq10271527}
		\begin{cases}
			(\sqrt{-\Delta+m^2 }-m )\varphi+i(v\cdot \nabla )\varphi+\lambda \varphi=\varphi\log|\varphi|^2+|\varphi|^{p-2}\varphi,
			&\text{in }\R^N,\\[3pt]
			\|\varphi\|_2^2=a,
		\end{cases}
	\end{equation}
	where the parameter $m,a>0$, $N\ge 2$, $\lambda\in\mathbb{R}$, and the travelling velocity $v\in\mathbb{R}^N$, $|v|<1$.
	
	Here, inspired by \cite{C83} and as in \cite{A17}, we consider the  reflexive Banach space
	$$W^{\frac12}(\R^N):=\left\{\varphi\in H^{\frac{1}{2}}(\R^N):|\varphi|^2\log|\varphi|^2\in L^1(\R^N)\right\}$$
	(see Section \ref{sec2} for more details). 
	Solutions to \eqref{eq10271527} can be obtained solving the minimization problem
	\begin{equation}\label{eq10271528}
		d_{p}(a):=\inf_{\varphi\in \mathcal{S}_a} E_{p}(\varphi ),
	\end{equation}
	where the energy functional $E_{p}(\varphi)$ is given by
	\begin{equation*}
		\begin{aligned}
			E_{p}(\varphi)
			&:=\frac{1}{2}\int_{\mathbb{R}^N}\bar{{\varphi}}(\sqrt{-\Delta+m^2 }-m+iv\cdot \nabla)  \varphi dx-\frac{1}{p} \|\varphi\|_p^{p}-\frac{1}{2}\inte|\varphi|^2(\log|\varphi|^2-1)dx,
		\end{aligned}
	\end{equation*}
	and the constraint $\mathcal{S}_a$ is defined as
	\begin{equation*}
		\mathcal{S}_a:=\left\{\varphi \in W^{\frac{1}{2}}(\mathbb{R}^N):\|\varphi\|_2^2=a\right\}.
	\end{equation*}

	Before stating our main results, we give some conventions and notations. We set
	\begin{equation}\label{eq01172054}
		a^\ast_v:=\|Q_v\|_2^2,
	\end{equation}
	where $Q_v\in H^{\frac{1}{2}}(\mathbb{R}^N)\setminus\{0\}$ is 
	an optimizer for inequality \eqref{eq11280935} and satisfies equation \eqref{eq11280944}. Observe that, as shown in \cite[Proposition 2.5]{DPSL26}, the number $a^\ast_v$ is well-defined. 
	
	Moreover, let
	$$\mathcal{X}_a:=\left\{u\in \mathcal{S}_a\cap C^{\infty}_c(\R^N, \mathbb R) : \inte|u|^2(\log|u|^2-1)dx\ge0 \right\},$$
	which is non-empty. Indeed, if $u\in\mathcal{S}_a$, setting $\phi_t:=t^{\frac{N}{2}}u(t\cdot)$, a simple calculation shows that
	\begin{equation*}
		\begin{aligned}
			\inte|\phi_t|^2(\log|\phi_t|^2-1)dx
			&=a(\log t^N-1)+\inte|u|^2\log|u|^2dx=:f(t),
		\end{aligned}
	\end{equation*}
	and
	$\lim\limits_{t\rightarrow+\infty}f(t)=+\infty$ and $\lim\limits_{t\rightarrow0^+}f(t)=-\infty.$ 
	Finally, for $2<p\le 2+\frac 2N$, let 
	\begin{equation}\label{mpstar}
		m_p^*(a):=\frac{p\sqrt{ 1-|v|^2}}{4}\inf_{u\in \mathcal{X}_a}\frac{\|\nabla u\|_2^2}{\|u\|_p^{p}}.
	\end{equation}
	We will see in Lemma \ref{le:m*} that $m_p^*(a)>0$ and is locally uniformly bounded with respect to $a$. This property plays an important role in the asymptotic analysis; see Remarks \ref{rem1.2} and \ref{rem}.

	Our first result can be stated as follows.
	\begin{thm}\label{Thm10291145}
		Suppose that $2<p<2+\frac{2}{N}$ with $N\ge 2$, and $m>m_p^*(a)$. Then, for any $a>0$, there exists at least a minimizer for problem (\ref{eq10271528}).
	\end{thm}
	
	Compared to Theorem \ref{Thm10291145}, our second result is about the existence and nonexistence of minimizers to problem (\ref{eq10271528}) in the case $p=2+\frac{2}{N}$.
	\begin{thm}\label{Thm01142024}
		Let $p=2+\frac{2}{N}$ with $N\ge 2$, then we have that
		\begin{enumerate}[label=(\arabic*),ref=\arabic*]
			\item \label{E} 
			If $a^\ast_v>a>0$ and $m>m_p^*(a)$, then problem (\ref{eq10271528}) has at least a minimizer;
			
			\item \label{NE} 
			If $a\ge a^\ast_v$ and $m>0$, then problem (\ref{eq10271528}) has no minimizers.
		\end{enumerate}

	\end{thm}
	
	Finally, for the case $p\in\left(2+\frac{2}{N},\frac{2N}{N-1}\right]$, we obtain the following results.
	\begin{thm}\label{Thm1.3}
		Let $m,a>0$, $2+\frac{2}{N}<p\le \frac{2N}{N-1}$ with $N\ge 2$. Then problem (\ref{eq10271528}) admits no minimizer.
	\end{thm}

	\begin{rem}
		As a byproduct, in the process of proving Theorems \ref{Thm10291145}, \ref{Thm01142024}  and \ref{Thm1.3}, we present the following energy estimates for $d_{p}(a)$. Specifically
		\begin{equation*}
			d_{p}(a)
			\left\{ \begin{array}{cll}
				< -\frac{1}{2}\left(1-\sqrt{1-|v|^2}\right )ma, & 2<p<2+\frac{2}{N},~ a>0,  m>m_p^*(a);\\[2mm]
				< -\frac{1}{2}\left(1-\sqrt{1-|v|^2}\right )ma, & p=2+\frac{2}{N},~ a^\ast_v>a>0, m>m_p^*(a);\\[2mm]
				=-\infty, & p=2+\frac{2}{N},~ a\ge a^\ast_v, m>0;\\[2mm]
				=-\infty, & 2+\frac{2}{N}<p\le \frac{2N}{N-1},~ a,m>0.
			\end{array}\right.
		\end{equation*}
	\end{rem}

	As discussed in Theorem \ref{Thm01142024}, we notice that the constant $a^\ast_v$ is a threshold for whether $d_{p}(a)$ is achieved or not.
	
	Inspired by \cite{DPSL26,GZ,HYZ25}, we next address the blow-up behavior of minimizers of $d_{p}(a)$ as $a\nearrow a^\ast_v$.
	\begin{thm}\label{Thm01152021}
		Set $p=2+\frac{2}{N}$ with $N\ge 2$. Consider a sequence $\{a_n\}$ with $a_n\nearrow a^\ast_v$ as $n\to \infty$, and take $m>0$ such that $d_p(a_n)$ is attained by a minimizer $u_{a_n}$ for all $n\ge 1$. Then, up to a subsequence,
		there exists a sequence $\{y_n\}\subset\mathbb{R}^N$ such that
		\begin{equation*}
			\epsilon _{n}^{\frac{N}{2} }u_{a_n}\left(\epsilon_{n}(\cdot+y_{n})\right)\to {(N\|Q_0\|_2^2)}^{-\frac N2}{Q_0\left(\frac{\cdot}{N\|Q_0\|_2^2}\right)}
			\text{~in~} H^{\frac{1}{2}}(\mathbb{R}^N),
		\end{equation*}
		where $Q_0$ optimizes the inequality (\ref{eq11280935}) and satisfies equation (\ref{eq11280944}), and
		$$\epsilon_{n}:=\left(\inte \bar{u}_{a_n}(\sqrt{-\Delta  }+iv\cdot \nabla )u_{a_n} dx\right)^{-1}\to  0.$$
	\end{thm}
	
	\begin{rem}\label{rem1.2}
		Observe that the existence of the family of minimizers $\{u_{a_n}\}$ in Theorem \ref{Thm01152021} is guaranteed by Theorem \ref{Thm01142024}-\eqref{E} and Lemma \ref{le:m*}, for suitably large $m$.
		Moreover, as $n\to \infty$, we establish the following energy estimates (see Lemma \ref{lem01190906}):
		$$d_{p}(a_n)\to-\infty,~~
		\epsilon_{n}d_{p}(a_n)\to 0,~~
		\epsilon_{n}\inte|u_{a_n}|^2\log|u_{a_n}|^2dx \to 0 ~\text{ and }~
		\epsilon_{n} \|u_{a_n}\|_{2+\frac{2}{N}}^{2+\frac{2}{N}}\to \frac{N+1}{N}.$$
	\end{rem}

	Finally, we show that the  asymptotic behavior of minimizers to problem (\ref{eq10271528}) in the case $2<p\le2+\frac{2}{N}$ with $N\ge 2$, which can be presented as follows.
	
	\begin{thm}\label{thm14.9}
		Let $\{a_n\}$ be a sequence satisfying
		\[
		\lim_{n\to \infty} a_n = a_0 \in
		\begin{cases}
			(0,+\infty), & \text{ if } 2<p<2+\frac{2}{N},\\
			(0,a_v^\ast), & \text{ if } p=2+\frac{2}{N},
		\end{cases}
		\]
		and suppose $m>0$ is such that $d_p(a_n)$ is attained by a minimizer $u_{a_n}$ for all $n\ge 1$.  
		Then, up to a subsequence, we have that
		$$u_{a_n}\to  {u}_{a_0}\text{~in~} W^{\frac{1}{2}}(\mathbb{R}^N), \text{ as } n\rightarrow\infty,$$
		where $ {u}_{a_0}$ is a minimizer related to $d_p(a_0)$.
	\end{thm}
	
	\begin{rem}\label{rem}
		Observations similar to those in Remark \ref{rem1.2} also apply to Theorem \ref{thm14.9}.
	\end{rem}
	
	The paper is organized as follows. In Section \ref{sec2}, we provide some preliminary results which are often used in the sequel.
	In Section \ref{sec3}, a full classification concerning the existence and nonexistence of minimizers to problem \eqref{eq10271528}  is established.
	Based on these conclusions, in Section \ref{sec5}, we analyse the precise blow-up behavior of minimizers of $d_p(a)$ as $a\nearrow a^\ast_v$.
	In Section \ref{sec6}, we further study the limiting profiles of minimizers of $d_p(a)$ as $a$ converges to some $a_0$.

	For simplicity of notations, we often use the abbreviations $L^p(\R^N)$ with norm $\|u\|_p=\|u\|_{L^p(\R^N)}$.
	Throughout this paper, the letters $C$, $C_i,~i=1,2,\ldots$ and so on will denote different positive constants from line to line.

	\section{Preliminaries}\label{sec2}
	For the readers' convenience, in this section, we briefly recall some basic definitions and important properties.

	The fractional Sobolev space $ H^{\frac{1}{2}}(\mathbb{R}^N)$ is defined by
	$$H^{\frac{1}{2}}(\mathbb{R}^N):=\left\{u\in L^2(\mathbb{R}^N): (-\Delta)^\frac{1}{4}u\in L^2(\mathbb{R}^N)\right\},$$
	and  endowed  with the norm
	$$\|u\|_{H^{\frac{1}{2}}(\mathbb{R}^N)}:=\left(\|u\|_2^2+\int_{\mathbb{R}^N}|k||\hat{u}(k)|^2dk\right)^\frac{1}{2}.$$
	In addition, for any velocity $v\in\R^N$ with  $|v|<1$, we define the quadratic form 
	$$T_v:\dot{H}^{\frac{1}{2} }(\R^N)\to \R,~~ u\mapsto T_v(u):=\inte \bar{u}(\sqrt{-\Delta  }+iv\cdot \nabla )u dx,$$
	where ${\dot{H}^{\frac{1}{2} }(\R^N)}$ denotes the usual Sobolev space equipped with the norm
	$$\|u\|_{\dot{H}^{\frac{1}{2} }(\R^N)}:=\| (-\Delta)^\frac{1}{4}u\|_{2}.$$
	Observe that 
	$$T_v(u)=\inte \bar{u}(\sqrt{-\Delta  }+iv\cdot \nabla )u dx=\inte (|k|-v\cdot k)|\hat{u}|^2dk.$$
	Clearly, we have the bounds
	\begin{equation}\label{eq10280938}
		(1-|v|)\|u\|_{\dot{H}^{\frac{1}{2} }(\R^N)}^2\leq T_v(u)\leq (1+|v|)\|u\|_{\dot{H}^{\frac{1}{2}}(\R^N)}^2,
	\end{equation}
	and so $\sqrt{T_v(u)}$ is a Hilbertian norm equivalent to the standard norm $\|u\|_{\dot{H}^{\frac{1}{2}}(\R^N)}$ since $|v|<1$.\\
	Let us consider also 
	\begin{equation*}
		\mathcal{T}_{m,v}(u):=\inte \bar{u}(\sqrt{-\Delta+m^2}+iv\cdot \nabla )udx=\inte (\sqrt{|k|^2+m^2}-v\cdot k)|\hat{u}|^2dk,  \text{ for each } u\in H^\frac{1}{2}(\mathbb{R}^N).
	\end{equation*}
	We have that $\sqrt{\mathcal{T}_{m,v}(u)}$ is a Hilbertian norm equivalent to $\|u\|_{{H}^{\frac{1}{2}}(\R^N)}$ (see \cite[Lemma A.4]{FJL07}) and, moreover, since
	\[
	\sqrt{|k|^2+m^2}-v\cdot k\ge \sqrt{1-|v|^2}m,\qquad\text{for all }k\in \R^N,
	\]
	we get
	\begin{equation}\label{APP.C}
		\mathcal{T}_{m,v}(u)\ge  \sqrt{1-|v|^2}m\|u\|_2^2, \qquad\text{for all }u\in H^{\frac12}(\R^N).
	\end{equation}
	
	A tool we will use is the following classical Lions lemma.
	
	\begin{lem}\label{lem10301100}(\cite[Lemma 2.2]{FQT12})
		Let $N\ge2$ and $R>0$. Assume that $\{u_n\}$ is a bounded sequence in $H^\frac{1}{2}(\R^N)$ and satisfies
		$$\lim_{n \to \infty}\sup _{y\in \mathbb{R}^N}\int_{B_R(y)}|u _{n} |^2dx =0, $$
		then
		$u_n\to 0$ in $L^q(\R^N)$ for any $2< q<\frac{2N}{N-1}.$
	\end{lem}
	
	Then we introduce the following Gagliardo-Nirenberg type inequality.
	\begin{lem}\label{lem10272157}(\cite[Proposition 3.1]{BGLV19} and \cite[Remark 2.3]{DPSL26})
		Let $2<q<\frac{2N}{N-1}$ with $N\ge2$ and $v\in \mathbb{R}^N$ with $|v|<1$. Then there exists a sharp constant $$\widetilde{C}_q=C_{v,N,q}:=\frac{q}{q-N(q-2)}\left[\left(\frac{q-N(q-2)}{N(q-2)}\right)^N\frac{1}{\|U_v\|_2^2}\right]^{\frac{q-2}{2}}>0,$$ 
		such that,  for all $u\in H^{\frac{1}{2} }(\mathbb{R}^N)$,
		\begin{equation}\label{eq10272204}
			\begin{aligned}
				\|u\|_q^{q}
				&\le \widetilde{C}_q\left(T_v(u)\right)^{\frac{N(q-2)}{2} }\|u\|_2^{q-N(q-2)},
			\end{aligned}
		\end{equation}
		where $U_v\in H^{\frac{1}{2} }(\mathbb{R}^N)\setminus\{0\}$ is an optimizer  for inequality \eqref{eq10272204} and satisfies
		\begin{equation}\label{eq01152114}
			(\sqrt{-\Delta  }+iv\cdot \nabla ) u  +u  =|u|^{q-2}u.
		\end{equation}
	\end{lem}
	
	By \cite[Lemma A.4]{BGLV19}, we further have
	\begin{lem}\label{lem2.5}
		Let $2<q<\frac{2N}{N-1}$ with $N\ge2$, and $v\in \mathbb{R}^N$ with $|v|<1$. Suppose that $U_v\in H^{\frac{1}{2} }(\mathbb{R}^N)$ solves equation \eqref{eq01152114}. 
		Then $U_v\in H^1(\mathbb{R}^N)\cap C_0(\mathbb{R}^N)$ and we have the decay estimate
		$$|U_v(x)|+|\nabla U_v(x)|\le \frac{C}{|x|^{N+1}}$$
		with some constant $C>0$. In particular,  we also have that $U_v\in L^\tau(\R^N)$, for all $\tau\in [1,+\infty]$, and $x\cdot\nabla U_v\in L^2(\mathbb{R}^N)$.
	\end{lem}
	
	In particular, if $q=2+\frac{2}{N}$ in Lemma \ref{lem10272157}, then the following results hold:
	\begin{lem}(\cite[Lemma 2.3]{LZW})
		Assume that $N\ge2$ and $v\in \mathbb{R}^N$ with $|v|<1$. Then for any $u\in H^{\frac{1}{2} }(\mathbb{R}^N)$, one has
		\begin{equation}\label{eq11280935}
			\begin{aligned}
				\|u\|_{2+\frac{2}{N}}^{2+\frac{2}{N}}
				&\le \frac{N+1}{N\|Q_v\|_2^{\frac{2}{N}}}T_v(u)\|u\|_2^{\frac{2}{N}}= \frac{N+1}{N{(a_v^*)}^{\frac{1}{N}}}T_v(u)\|u\|_2^{\frac{2}{N}},
			\end{aligned}
		\end{equation}
		where $a_v^*$ is defined in \eqref{eq01172054}. 
		Here, the equality holds if $u=Q_v$, and $Q_v\in H^{\frac{1}{2} }(\mathbb{R}^N)\setminus\{0\}$ is a solution of the equation
		\begin{equation}\label{eq11280944}
			(\sqrt{-\Delta  }+iv\cdot \nabla ) u  +u  =|u|^{\frac{2}{N}}u.
		\end{equation}
	\end{lem}
	
	Now, as in \cite{C83}, set
	\begin{equation*}
		A(s):=
		\begin{cases}
			-s^2\log s^2,   \,\ &0\leq s\leq e^{-3},\\
			3s^2+4e^{-3}s-e^{-6},   \,\ &s\geq e^{-3},
		\end{cases}
	\end{equation*}
	and
	$$B(s):=s^2\log s^2+A(s).$$
	
	The functions $A,B$ are non-negative, convex and increasing functions on $[0,+\infty)$, 
	and, for all $2<q$, there exists a constant $C_q>0$ such that
	\begin{equation}\label{eq01161709}
		0\le B(s)\le C_q|s|^q.
	\end{equation}
	
	Denote
	$$L^A(\R^N):=\left\{u\in L_{\rm loc}^{1}(\R^N):A(|u|)\in L^{1}(\R^N)\right\},$$
	which is the Orlicz space corresponding to $A$ and equip with the following Luxemburg norm
	$$\|u\|_{L^A(\R^N)}:=\inf \left\{k>0:\inte A\left(k^{-1}|u(x)|\right)dx\leq 1\right\}.$$
	Then (see \cite[Lemma 2.2]{A17})
	$$W^{\frac12}(\R^N)=H^{\frac12}(\R^N)\cap{L^A}(\R^N)
	$$
	and we equip it with the norm
	$$\|\varphi\|_{W^{\frac12}(\R^N)}:=\|\varphi\|_{H^{\frac12}(\R^N)}+\|\varphi\|_{L^A(\R^N)} .$$

	The following properties hold.
	\begin{lem}\label{lem10261928}(\cite[Lemma 2.1]{C83})
		Let $N\ge2$ and $\{u_n\}$ be a sequence in $L^A(\R^N)$. Then, the following facts hold.
		\begin{itemize}
			\item [(i)] If $u_n\to u$ in $L^A(\R^N)$, then $A(|u_n|)\to A(|u|)$ in $L^1(\R^N)$ as $n\to\infty$.
			\item [(ii)] Let $u\in L^A(\R^N)$. If $u_n\to u$ a.e. in $\R^N$ and
			$$\lim_{n\to\infty}\inte A(|u_n|)dx= \inte A(|u|)dx,$$
			then $u_n\to u$ in $L^A(\R^N)$ as $n\to\infty$.
			\item [(iii)] For any $u\in L^A(\R^N)$, we have that
			\begin{equation}\label{A10261950}
				\min\left\{\|u\|_{L^A(\R^N)},\|u\|_{L^A(\R^N)}^2\right\}\leq \inte A(|u|)dx \leq \max\left\{\|u\|_{L^A(\R^N)},\|u\|_{L^A(\R^N)}^2\right\}.
			\end{equation}
		\end{itemize}
	\end{lem}

	\begin{lem}\label{lem10260927}(\cite[Lemma 2.3]{A17})
		Let $N\ge2$ and $\{u_n\}$ be a bounded sequence in $W^{\frac12}(\R^N)$ such that $u_n \to u $ a.e. in $\R^N$. Then we have that $u\in W^{\frac12}(\R^N)$ and
		\begin{equation*}
			\lim\limits_{n\to\infty}\inte \left(|u_{n}|^2\log|u_{n}|^2-|u_{n}-u|^2\log|u_{n}-u|^2\right)dx=\inte |u|^2\log|u|^2dx.
		\end{equation*}
	\end{lem}

	Moreover, we recall the following well-known {\em logarithmic Sobolev inequality}.
	\begin{lem}(\cite[Theorem 8.14]{LL})
		Let $u\in H^1(\R^N)$ and let $b>0$ be any number. Then 
		\begin{equation}\label{log-sob}
			\frac{b^2}{\pi} \|\nabla u\|_2^2\ge \inte |u|^2\log\left(\frac{|u|^2}{\|u\|^2_2}\right)dx
			+N(1+\log b)\|u\|_2^2.
		\end{equation}
	\end{lem}

	We also have the following locally uniform bound for $m_p^*(a)$ with respect to $a$.
	
	\begin{lem}\label{le:m*}
		Let $2<p\le 2+\frac 2N$ and $0<\underline{a}<\overline{a}<+\infty$. Then there exist $\underline{m}=\underline{m}(p,\overline{a})$ and $\overline{m}=\overline{m}(p,\underline{a})$ with $0<\underline{m}<\overline{m}<+\infty$ such that $m_p^*(a)\in [\underline{m},\overline{m}]$, for all $a\in [\underline{a},\overline{a}]$.
	\end{lem}
	
	\begin{proof}
		Fix $p>2$ and $0<\underline{a}<\overline{a}<+\infty$, and take any $a\in [\underline{a},\overline{a}]$. 
		
		Let's show a uniform bound from below.
		Fix $u\in \X_a$. By the well-known Gagliardo-Nirenberg inequality
		$$\|v\|_p^p\le C\|\nabla v\|_2^{\frac{N(p-2)}{2}}\|v\|_2^{p-\frac{N(p-2)}{2}}, \text{ for all } v\in H^1(\R^N),$$
		using \eqref{log-sob},
		and recalling the definition of the set $\X_a$, there exists a constant $C_b>0$ such that 
		\begin{align*}
			\frac{\|\nabla u\|_2^2}{\|u\|_p^{p}}  
			\ge C \frac{\|\nabla u\|_2^{2-\frac{N(p-2)}{2}}}{a^\frac{2p-N(p-2)}{4}}
			&\ge C_b \left[\inte |u|^2(\log |u|^2-1)dx+a-a\log a+N(1+\log b)a\right]^{1-\frac{N(p-2)}{4}}a^\frac{N(p-2)-2p}{4}
			\\
			&\ge C_b \left[1-\log a+N(1+\log b)\right]^{1-\frac{N(p-2)}{4}}a^{1-\frac{p}{2}}\\
			&\ge C_b \left[1-\log \overline{a}+N(1+\log b)\right]^{1-\frac{N(p-2)}{4}}\overline{a}^{1-\frac{p}{2}},
		\end{align*}
		where $b> \operatorname{exp}\left(\frac{\log \overline{a}-1}{N}-1\right)$.
		Therefore,
		by the arbitrariness of $u\in \X_a$ we conclude.

		We now prove a uniform bound from above. Let $u\in \X_1$ and, for any $t>0$, take $\varphi_t:=\sqrt{a}\phi_t$, where $\phi_t:=t^{\frac N2}u(t\cdot)$. Clearly $\varphi_t\in \mathcal{S}_a$, for any $t>0$, but not necessarily in $\X_a$. Since
		\begin{align*}
			\inte |\varphi_t|^2\left(\log|\varphi_t|^2-1 \right)dx
			&=a\log a +a \inte |\phi_t|^2\left(\log|\phi_t|^2-1 \right)dx
			\\
			&=a\left(\log a+\log t^N-1+\inte |u|^2\log|u|^2dx\right),   
		\end{align*}
		we can find that $\varphi_{t}\in \X_a$ for every
		
		\[
		t\geq t_a:=\operatorname{exp}\left[\frac{1}{N}\left(1-\log a - \inte |u|^2\log|u|^2dx\right)\right].
		\]
		Since the map $a\mapsto t_a$ is decreasing, then, for every $a\in [\underline{a},\overline{a}]$ and $t\geq t_{\underline{a}}$, $\varphi_{t}\in \X_a$ and so
		\[
		\sup_{a\in [\underline{a},\overline{a}]}m_p^*(a)
		\le C \frac{\|\nabla \varphi_{t_{\underline{a}}} \|_2^2}{\|\varphi_{t_{\underline{a}}}\|_p^{p}}
		=  C \underline{a}^{1-\frac p2}t_{\underline{a}}^{2+N-\frac{Np}{2}}\frac{\|\nabla u \|_2^2}{\|u\|_p^{p}}. 
		\]
	\end{proof}

	We conclude with some properties of $Q_v$. It is easy to check that $Q_v\in W^{\frac{1}{2} }(\mathbb{R}^N)$. In fact, by Lemma \ref{lem2.5}, we have that
	\begin{equation}\label{eq10280849}
		\begin{aligned}
			\left|\inte|Q_v|^2\log|Q_v|^2dx\right|
			&\le\inte|Q_v|^2|\log|Q_v|^2|dx\\
			&\leq C_\delta\left(\inte|Q_v|^{2-\delta}dx+\inte|Q_v|^{2+\frac{2}{N}}dx\right)
			\le C,~ \text{ where } 0<\delta<1.
		\end{aligned}
	\end{equation}
	Then, from \eqref{eq01161709}, it follows that
	\begin{equation*}
		\begin{aligned}
			\inte A(|Q_v|)dx
			&=\inte B(|Q_v|)dx-\inte|Q_v|^2\log|Q_v|^2dx
			\le  C_{p}\|Q_v\|_{2+\frac{2}{N}}^{2+\frac{2}{N}} +C\le  C,
		\end{aligned}
	\end{equation*}
	which, together with Lemma \ref{lem10261928}, yields that $Q_v\in L^{A}(\mathbb{R}^N)$.
	
	In addition, by \cite[Lemma A.3]{BGLV19} and \eqref{eq01172054}, one can see that
	\begin{equation}\label{eq01071942}
		T_v(Q_v)=\frac{N}{N+1} \|Q_v\|_{2+\frac2N}^{2+\frac2N}
		=N\|Q_v\|_2^2=Na^\ast_v.
	\end{equation}

	\section{Existence and nonexistence of minimizers}\label{sec3}
	In this section, our goal is to establish both the existence and nonexistence of minimizers for the minimization problem \eqref{eq10271528}, as stated in Theorems \ref{Thm10291145}, \ref{Thm01142024}, and \ref{Thm1.3}. According to the influence of the exponent $p\in\left(2,\frac{2N}{N-1}\right]$ on the geometric structure of the energy functional $E_{p}$, the proof is divided into the following three subsections.

	\subsection{Existence of minimizers for the case \texorpdfstring{$2<p<2+\frac{2}{N}$}{}}
	In this subsection, we intend to prove Theorem \ref{Thm10291145}, namely,  the existence of minimizers to problem \eqref{eq10271528} in the case $2<p<2+\frac{2}{N}$. For this purpose, we begin by showing that, under suitable conditions, the minimization problem \eqref{eq10271528} is well-defined.
	
	\begin{prop}\label{prop01152155}
		Let $m,a>0$, $2<p<2+\frac{2}{N}$ with $N\ge 2$.
		Then $E_{p}$ is bounded from below on $\mathcal{S}_a$.
		
	\end{prop}
	
	\begin{proof}
		Noting that the operator inequality
		\begin{equation}\label{eq102881138}
			\sqrt{-\Delta} -m\le \sqrt{-\Delta+m^2} -m\le \sqrt{-\Delta},
		\end{equation}
		then, for any $u \in \mathcal{S}_a$, from \eqref{eq01161709} and \eqref{eq10272204}  for $q=p$, we deduce
		\begin{equation}\label{eq10281212}
			\begin{aligned}
				E_{p}(u)
				&\geq \frac{1}{2}T_v(u)-\frac{ma}{2}-\frac{1}{p} \|u\|_p^{p}-\frac{1}{2}\inte [B(|u|)- A(|u|)]dx+\frac{a}{2} \\
				&\ge \frac{1}{2}T_v(u)-\frac{ma}{2}-C
				\left(T_v(u)\right)^{\frac{N(p-2)}{2} }a^{\frac{p}{2}-\frac{N(p-2)}{2} }  >C
			\end{aligned}
		\end{equation}
		for a suitable $C\in\R$.
	\end{proof}

	Furthermore, if $m>m_p^*(a)$, where $m_p^*(a)$ is defined in \eqref{mpstar},  we have the following upper bound of $d_{p}(a)$.

	\begin{lem}\label{lem10281516}
		Let $a>0$, $2<p<2+\frac{2}{N}$ with $N\ge2$, and $m>m_p^*(a)$. Then we have that
		\begin{equation}\label{eq10281747}
			d_{p}(a)< -\frac{1}{2}\left(1-\sqrt{1-|v|^2}\right )ma.
		\end{equation}
	\end{lem}
	
	\begin{proof}
		Without loss of generality, here and in what follows, we always assume that  $v$ is parallel to the $x_N-axis$, i.e., $v=|v|e_{x_N}$. 
		Let  $\phi\in \mathcal{X}_a$ and introduce 
		\begin{equation}\label{eq10281521}
			\phi _{\lambda }(x):=e^{i\lambda v\cdot x}\phi (x)=e^{i\lambda |v|{x_N}}\phi (x),~~\text{ where } \lambda>0.
		\end{equation}
		Since 
		$$\int_{\mathbb{R}^N}\phi\frac{\partial\phi}{\partial {x_i}}dx=0,$$ 
		we get that
		\begin{equation}\label{eq10281526}
			\frac{i}{2}\int_{\mathbb{R}^N}\bar{\phi} _{\lambda }(v\cdot \nabla )\phi _{\lambda }dx=-\frac{\lambda |v|^2}{2} a
			\ \text{ and }\ 
			\|\phi_{\lambda}\|_2^2=\|\phi\|_2^2=a.
		\end{equation}
		By using \eqref{eq10281521} and \eqref{eq10281526}, we have that
		\begin{equation}\label{eq10281528}
			\begin{aligned}
				E_{p}(\phi _{\lambda })
				\le\frac{1}{2}\left(\int_{\mathbb{R}^N}\bar{\phi}_{\lambda }(\sqrt{-\Delta+m^2} -m)\phi _{\lambda }dx-\lambda |v|^2a\right)
				-\frac{1}{p}\|\phi\|_p^{p}=:I_1-\frac{1}{p} \|\phi\|_p^{p}.
			\end{aligned}
		\end{equation}
		Moreover, from the operator inequality
		$$\sqrt{-\Delta+m^2}\le \frac{1}{2\lambda } (-\Delta+m^2+\lambda^2),$$
		it follows that
		\begin{equation}\label{eq10281534}
			\begin{aligned}
				I_1
				&\leq\frac{1}{4\lambda }\int_{\mathbb{R}^N}\bar{\phi}_{\lambda }({-\Delta+m^2} +{\lambda }^2)\phi _{\lambda }dx-\frac{ma}{2}-\frac{\lambda |v|^2a}{2} \\
				&=\frac{1}{4\lambda}\|\nabla \phi\|_2^2
				+\frac{(1-|v|^2)\lambda^2+m^2-2m\lambda}{4\lambda}a
				=:\frac{1}{4\lambda}\|\nabla \phi\|_2^2+f(\lambda) .
			\end{aligned}
		\end{equation}
		Then, taking the minimizer $\lambda_*=m(1-|v|^2)^{-1/2}$ of $f$ in $(0,\infty)$, from \eqref{eq10281528} and \eqref{eq10281534} we conclude that
		\begin{equation}\label{eq10281543}
			\begin{aligned}
				E_{p}(\phi _{\lambda_*})
				&\leq f(\lambda_*)+\frac{1}{4\lambda_*}\|\nabla \phi\|_2^2-\frac{1}{p} \|\phi\|_p^{p}\\
				&=-\frac{1}{2}\left(1-\sqrt{ 1-|v|^2}\right)ma +\frac{\sqrt{ 1-|v|^2}}{4m} \|\nabla \phi\|_2^2-\frac{1}{p} \|\phi\|_p^{p}\\
				&=:-\frac{1}{2}\left(1-\sqrt{ 1-|v|^2}\right)ma + \tilde{E}_{p}(\phi).
			\end{aligned}
		\end{equation}
		In addition, let $\epsilon>0$ be such that $m_p^*(a)<m_p^*(a)+\epsilon<m$. For such $\epsilon$, there exists $\phi_\epsilon\in \mathcal{X}_a$ such that
		$$m_p^*(a)\le \frac{p\sqrt{ 1-|v|^2}\|\nabla \phi_\epsilon\|_2^2}{4\|\phi_\epsilon\|_p^{p}}< m_p^*(a)+\epsilon.$$
		Then, we get
		\begin{equation*}
			\tilde{E}_{p}(\phi_\epsilon)
			=\frac{\|\phi_\epsilon\|_p^{p}}{p}\left(\frac{p\sqrt{ 1-|v|^2}\|\nabla \phi_\epsilon\|_2^2}{4\|\phi_\epsilon\|_p^{p}m} -1\right) 
			<\frac{\|\phi_\epsilon\|_p^{p}}{p}\left(\frac{m_p^*(a)+\epsilon}{m} -1\right)<0.
		\end{equation*}
		Hence, by \eqref{eq10281543}, we derive \eqref{eq10281747}.
	\end{proof}

	Based on Lemma \ref{lem10281516}, we now establish a strict subadditivity inequality and some properties about $d_{p}(a)$. For this aim, for any $a>0$, we define the auxiliary variational problem
	\begin{equation}\label{eq10281724}
		d_{p}^{a}(1):=\inf_{\psi \in \mathcal{S}_1}{E}_{p}^{a} (\psi),
	\end{equation}
	where
	\begin{equation}\label{eq01142202}
		\begin{split}
			{E}_{p}^{a} (\psi)
			&:=\frac{1}{2}\int_{\mathbb{R}^N}\bar{\psi}(\sqrt{-\Delta+m^2 }-m +iv\cdot \nabla)  \psi dx-\frac{a^{\frac{p-2}{2}}}{p} \|\psi\|_p^{p}-\frac{\log a}{2}\|\psi\|_2^2\\
			&~~~~-\frac{1}{2}\inte|\psi|^2(\log|\psi|^2-1)dx.
		\end{split}
	\end{equation}
	By a similar argument as Proposition \ref{prop01152155}, one can see that ${E}_{p}^{a}$ is uniformly bounded from below on $\mathcal{S}_1$. Meanwhile, we notice that the fact that $u_a$ is a minimizer of $d_{p}^{a}(1)$ if and only if $\sqrt{a}u_a$ is a minimizer of $d_{p}(a)$, and 
	\begin{equation}\label{eq10282013}
		d_{p}(a)=ad_{p}^{a}(1).
	\end{equation}
	Then, there holds
	\begin{lem}\label{lem10281714}
		Under the assumptions of Lemma \ref{lem10281516}, one has
		\begin{equation}\label{eq10291042}
			d_{p}(a)<d_{p}(\lambda )+d_{p}(a-\lambda ), \text{ where } 0<\lambda<a.
		\end{equation}
		Moreover, the map $t\in (0,a] \mapsto d_{p}(t)$ is strictly decreasing and continuous.
	\end{lem}
	
	\begin{proof}
		We prove the conclusions by dividing three steps.
		
		{\emph{Step 1: the map $t\in (0,a] \mapsto d_{p}(t)$ is strictly decreasing.}} 
		Choosing $0< t_1<t_2\leq a$ and taking $\{\psi_n^{t_1}\} \subset\mathcal{S}_1$ a minimizing sequence for problem \eqref{eq10281724} with $a=t_1$, we have that
		\begin{equation*}
			\begin{aligned}
				d_{p}^{t_1}(1)
				&=\lim_{n \to \infty} {E}_{p}^{t_1}(\psi_n^{t_1})
				= \lim_{n \to \infty}\left[
				{E}_{p}^{t_2}(\psi_n^{t_1})
				+\frac{t_2^{\frac{p-2}{2}}-t_1^{\frac{p-2}{2}}}{p} \|\psi_n^{t_1}\|_p^{p}
				+\frac{\log t_2-\log t_1}{2}\right]\\
				&> \liminf_{n \to \infty}{E}_{p}^{t_2}(\psi_n^{t_1})\geq d_{p}^{t_2}(1),
			\end{aligned}
		\end{equation*}
		which implies that the map $t\in (0, a] \mapsto d_{p}^{t}(1)$ is strictly decreasing.
		Moreover, for every $a>0$, we derive from \eqref{eq10281747} and (\ref{eq10282013}) that
		\begin{equation}\label{eq10291114}
			d_{p}^{a}(1)=\frac{1}{a}d_{p}(a)< -\frac{1}{2}\left(1-\sqrt{1-|v|^2} \right)m<0.
		\end{equation}
		Then,
		$$d_{p}(t_1)=t_1d_{p}^{t_1}(1) >t_2d_{p}^{t_1}(1)>t_2d_{p}^{t_2}(1)=d_{p}(t_2).$$
		
		{\emph{Step 2: The subadditivity inequality \eqref{eq10291042} holds.}} By using (\ref{eq10282013}) and
		the fact that the map $t\in (0, a] \mapsto d_{p}^{t}(1)$ is strictly decreasing, we infer  that
		$$d_{p}(\theta M)=\theta Md_{p}^{\theta M}(1) <\theta Md_{p}^{ M}(1)=\theta d_{p}( M),$$
		where $a> M>0,~ a/M\geq\theta>1$. By 
		\cite[Lemma II.1]{Lion}, this inequality leads to the strict subadditivity inequality  \eqref{eq10291042}.
		
		{\emph{Step 3: the map $t\in (0,a] \mapsto d_{p}(t)$ is continuous.}} By applying \eqref{eq10282013}, here it is sufficient to prove that the map $t\in (0,a] \mapsto d_{p}^{t}(1)$ is continuous.
		
		For any $t_1,t_2\in (0,a]$ with $t_1<t_2$, by the definition of $d_p^{t}(1)$, there exists  $\psi\in\mathcal{S}_1$ such that
		\begin{equation}\label{eq-inf}
			d_p^{t_2}(1)\le {E}_p^{t_2}(\psi)< d_p^{t_2}(1)+t_2-t_1< d_p^{t_2}(1)+a<a,
		\end{equation}
		where we have used \eqref{eq10291114}.
		Arguing as in \eqref{eq10281212}, one has
		\begin{equation*}
			\begin{aligned}
				E_p^{t_2}(\psi)
				&\ge \frac{1}{2}T_v(\psi)-\frac{m}{2}-\frac{\widetilde{C}_p t_2^{\frac{p-2}{2}}}{p}\left(T_v(\psi)\right)^{\frac{N(p-2)}{2}} -\frac{C_{p}\widetilde{C}_{p}}{2}\left(T_v(\psi)\right)^{\frac{N(p-2)}{2}}+\frac{1 }{2}-\frac{ \log t_2}{2} \\
				&\ge \frac{1}{2}T_v(\psi)-\frac{m}{2}-\frac{\widetilde{C}_p a^{\frac{p-2}{2}}}{p}\left(T_v(\psi)\right)^{\frac{N(p-2)}{2}} -\frac{C_{p} \widetilde{C}_{p}}{2}\left(T_v(\psi)\right)^{\frac{N(p-2)}{2}}-\frac{ \log a}{2},
			\end{aligned}
		\end{equation*}
		which, together with \eqref{eq-inf}, implies that  $\psi$ is uniformly bounded in $H^{\frac{1}{2}}(\R^N)$ with respect to $t_1$ and $t_2$. 
		According to the fact that the map $t\in (0,a] \mapsto d_{p}^{t}(1)$ is strictly decreasing, by \eqref{eq-inf}  again, it then follows that
		\begin{equation*}
			\begin{aligned}
				0<d_p^{t_1}(1)-d_p^{t_2}(1)
				&\le {E}_p^{t_1}(\psi)-{E}_p^{t_2}(\psi)+t_2-t_1\\
				&=\frac{1}{p}\left(t_2^{\frac{p-2}{2}}-t_1^{\frac{p-2}{2}}\right) \|\psi\|_p^{p}+\frac{\log t_2-\log t_1}{2}+t_2-t_1\\
				&\leq C_1\left(t_2^{\frac{p-2}{2}}-t_1^{\frac{p-2}{2}}\right) +\frac{\log t_2-\log t_1}{2}+t_2-t_1,
			\end{aligned}
		\end{equation*}
		and	the proof is completed.
	\end{proof}

	With the help of the above conclusions, we now prove Theorem \ref{Thm10291145} by applying the concentration-compactness principle \cite{Lion}, arguing as in \cite{LW}.

	\begin{proof}[Proof of Theorem \ref{Thm10291145}]
		Taking any minimizing sequence $\{\psi_k\} \subset\mathcal{S}_a$ so that  $\lim\limits_{k \to \infty} E_{p}(\psi_k)=d_{p}(a)$, we claim that $\{\psi_k\} $ is bounded in $H^{\frac{1}{2}}(\mathbb{R}^N)$ for $2<p<2+\frac{2}{N}$.
		In fact, if $\{\psi_k\}$ is unbounded in $H^{\frac{1}{2}}(\mathbb{R}^N)$, by \eqref{eq10280938}, \eqref{eq10281212}, and \eqref{eq10281747}, we immediately derive  that
		\begin{equation*}
			0>d_{p}(a)
			= E_{p}(\psi_k)+o_k(1)
			\geq
			\frac{1}{2}T_v(\psi_k)-\frac{ma}{2}
			-C a^{\frac{p}{2}-\frac{N(p-2)}{2}}
			\left(T_v(\psi_k)\right)^{\frac{N(p-2)}{2} }+o_k(1)\to +\infty,
		\end{equation*}
		which is a contradiction.
		
		Furthermore, we prove that $\{\psi_k\}$ is bounded in $L^A(\R^N)$. By \eqref{eq10280938}, \eqref{eq10272204} and \eqref{eq102881138}, we infer that
		\begin{equation*}
			\begin{aligned}
				&d_{p}(a)= E_{p}(\psi_k)+o_k(1)\\
				&\geq \frac{1}{2}T_v(\psi_k)-\frac{ma}{2}-\frac{\widetilde{C}_p}{p}\left(T_v(\psi_k)\right)^{\frac{N(p-2)}{2} }a^{\frac{p}{2}-\frac{N(p-2)}{2} }-\frac{1}{2}\inte|\psi_k|^2\log|\psi_k|^2dx+o_k(1)\\
				&\ge-\frac{ma}{2}-\frac{\widetilde{C}_p}{p}\left((1+|v|)\|\psi_k\|_{\dot{H}^{\frac{1}{2}}(\R^N)}^2\right)^{\frac{N(p-2)}{2} }a^{\frac{p}{2}
					-\frac{N(p-2)}{2} }-\frac{1}{2}\inte|\psi_k|^2\log|\psi_k|^2dx+o_k(1).
			\end{aligned}
		\end{equation*}
		Combining the fact that $\{\psi_k\}$ is bounded in $H^{\frac{1}{2}}(\mathbb{R}^N)$, it yields that
		\begin{equation}\label{eq10291746}
			\begin{aligned}
				-\inte|\psi_k|^2\log|\psi_k|^2dx\le C+o_k(1).
			\end{aligned}
		\end{equation}
		Thus, by \eqref{eq01161709} 
		and \eqref{eq10291746}, we get that
		\begin{equation*}
			\begin{aligned}
				\inte A(|\psi_k|)dx
				=\inte (B(|\psi_k|)-|\psi_k|^2\log |\psi_k|^2)dx\leq C+o_k(1),
			\end{aligned}
		\end{equation*}
		which, together with \eqref{A10261950}, indicates that $\{\psi_k\}$ is bounded in $L^A(\R^N)$.
		
		To sum up, we conclude that $\{\psi_k\}$ is bounded in $W^{\frac{1}{2}}(\mathbb{R}^N)$.
		
		Now, let us introduce the following L\'{e}vy concentration function
		\begin{equation*}
			Q_k(r):=\sup_{y\in\R^N}\int_{B_r(y)}|\psi_k(x)|^2dx.
		\end{equation*}
		Since $\{Q_k\}$ is sequence of monotone and uniformly bounded functions, by Helly's selection theorem,  we can find a convergent subsequence, denoted again by $\{Q_k\}$, such that there is a non-decreasing function $Q(r)$ satisfying
		$$\lim_{k\to \infty}Q_k(r)=Q(r), \text{ for all } r>0.$$
		Noting that $0\le Q_k(r)\le a$, there exists $\beta\in [0,a]$ such that
		\begin{equation}\label{eq1.2}
			\lim_{r\to +\infty}Q(r)=\beta.
		\end{equation}

		Then,  we will divide the proof into two cases as follows:
		
		{\em Case 1: $\beta=0$.} Recalling the definition of $\beta$ in \eqref{eq1.2}, we have that
		\begin{equation}\label{eq14.19}
			\lim_{k \to \infty}\sup _{y\in \mathbb{R}^N}\int_{B_R(y)}|\psi_{k} |^2dx =0, \mbox{ for all }R>0, 
		\end{equation}
		which, together with Lemma \ref{lem10301100}, indicates that
		\begin{equation}\label{eq10301550}
			\lim_{k \to \infty}\|\psi_k\|_\tau^\tau =0,~~\text{ for all } 2< \tau<2N/(N-1).
		\end{equation}
		Set $\Omega_k:=\{x\in\R^N:|\psi_{k}(x)|\ge 1\}$. If $\Omega_k=\emptyset$ for infinitely many $k\ge1$, then clearly for such $k$
		\begin{equation}\label{empty}
			\inte|\psi_k|^2\log |\psi_k|^2dx\le 0.
		\end{equation}
		Otherwise, up to a subsequence, it follows from \eqref{eq10301550} that
		\begin{equation}\label{eq10301551}
			\inte|\psi_k|^2\log |\psi_k|^2dx
			\le\int_{\Omega_k} |\psi_k|^2\log |\psi_k|^2dx
			\leq C\int_{\Omega_k}|\psi_k|^{2+\frac{2}{N}}dx\leq  C\|\psi_k\|_{2+\frac{2}{N}}^{2+\frac{2}{N}}=o_k(1).
		\end{equation}
		Moreover, by \eqref{APP.C}, one has
		\begin{equation*}
			\inte\bar{\psi}(\sqrt{-\Delta+m^2 }-m +iv\cdot \nabla)  \psi dx \ge -\left(1-\sqrt{1-|v|^2} \right)ma,~~\forall\psi\in \mathcal{S}_a.
		\end{equation*}
		Combining with \eqref{eq10301550}, \eqref{empty}, \eqref{eq10301551}, and Lemma \ref{lem10281516}, this gives
		\begin{equation*}
			\begin{aligned}
				-\frac{1}{2}\left(1-\sqrt{1-|v|^2}\right )ma
				>d_{p}(a)=\lim\limits_{k \to \infty} E_{p}(\psi_k)
				>-\frac{1}{2}\left(1-\sqrt{1-|v|^2} \right)ma,
			\end{aligned}
		\end{equation*}
		which is absurd. Hence, the case $\beta=0$ does not hold.
		
		{\em Case 2: $\beta\neq0$.} In this case, for large enough $R>0$, \eqref{eq1.2} gives that $\frac{\beta}{2}<Q(R)<\frac{3\beta}{2}$. By passing a subsequence if necessary, we have that
		\begin{equation*}
			\frac{\beta}{2}\le\lim_{k \to \infty}\sup _{y\in \mathbb{R}^N}\int_{B_R(y)}|\psi_{k} |^2dx\le\frac{3\beta}{2}.
		\end{equation*}
		Then
		there exists some $\{y_k\}\subset\R^N$ such that
		\begin{equation}\label{eq1.3}
			\frac{\beta}{2}\le\lim_{k \to \infty}\int_{B_R(0)}|\psi_{k}(x+y_k)|^2dx\le\frac{3\beta}{2}.
		\end{equation}
		Noting that the fact $\{\psi_k(\cdot+y_k)\}$ is bounded in $W^{\frac{1}{2}}(\mathbb{R}^N)$, up to the subsequence, then there exists $\varphi \in W^{\frac{1}{2}}(\mathbb{R}^N)$ such that, as $k\to \infty$,
		\[
		\begin{array}{lll}
			\psi_k(\cdot+y_k)\rightharpoonup\varphi \text{ weakly in } W^{\frac{1}{2}}(\mathbb{R}^N),
			&
			&
			\psi_k(\cdot+y_k)\rightharpoonup\varphi \text{ weakly in } H^{\frac{1}{2}}(\mathbb{R}^N),\\
			\psi_k(\cdot+y_k)\to\varphi \text{ in } L_{\rm loc}^{2}(\mathbb{R}^N),
			&
			&
			\psi_k(\cdot+y_k)\to\varphi \text{ a.e. in } \mathbb{R}^N.
		\end{array}
		\]
		Therefore, from \eqref{eq1.3} we deduce that
		$\|\varphi\|_2^2\ge \frac{\beta}{2}>0.$
		
		Set $v_k:=\psi_k(\cdot+y_k)-\varphi$.
		By 
		Brezis-Lieb Lemma \ref{lem10260927}, we have that
		\begin{equation}\label{eq1.5}
			a=\|\psi_k\|_2^2=\|v_k\|_2^2+\|\varphi\|_2^2+o_k(1)
		\end{equation}
		and
		\begin{equation}\label{eq1.4}
			E_{p}(\psi_k)=E_{p}(v_k)+ E_{p}(\varphi)+o_k(1).
		\end{equation}
		
		Let $0<\beta_1:=\|\varphi\|_2^2$
		and $\delta:=a-\beta_1$. Then \eqref{eq1.5} gives that
		$$\lim_{k \to \infty}\|v_k\|_2^2=\delta\ge 0.$$
		
		If $\delta >0$, taking $\tilde{v}_k:=b_kv_k$ and $b_k:=\frac{\sqrt{\delta}}{\|v_k\|_2}$, then
		$\|\tilde{v}_k\|_2^2=\delta$ and $\lim\limits_{k\to \infty}b_k=1.$
		For $k$ big enough, one has
		\begin{equation}\label{eq1.10}
			d_p(\delta)\le E_{p}(\tilde{v}_k)=E_{p}(v_k)+o_k(1).
		\end{equation}
		Hence, from \eqref{eq1.4} and \eqref{eq1.10}, we derive that
		\begin{equation*}
			\begin{aligned}
				d_p(a)=E_{p}(\psi_k)+o_k(1)
				= E_{p}(\varphi)+E_{p}(v_k)+o_k(1)
				= E_{p}(\varphi)+E_{p}(\tilde{v}_k)+o_k(1)
				\ge d_p(\beta_1)+d_{p}(\delta),
			\end{aligned}
		\end{equation*}
		which leads to a contradiction by using 
		Lemma \ref{lem10281714}. Thus, we get $\delta =0$, that is,
		$$\lim_{k\to \infty}\|\psi_k\|_2^2=\|\varphi\|_2^2=a.$$
		Using  interpolation inequality, this means that
		\begin{equation}\label{eq10302001}
			\psi_k(\cdot +y_k)\to \varphi \text{~in}~ L^\gamma(\mathbb{R}^N),  ~\forall \gamma\in\left[2,2N/(N-1)\right).
		\end{equation}
		
		On the other hand, recalling  \cite[inequality (2.4)]{A17}
		$$\inte|B(|u|)-B(|v|)|dx\le C\left(\|u\|_{H^{\frac{1}{2}}(\mathbb{R}^N)}+\|v\|_{H^{\frac{1}{2}}(\mathbb{R}^N)}\right)^{\frac{N}{N-1}}\|u-v\|_{2},\qquad\text{for all }u,v\in H^{\frac 12}(\R^N),$$
		and using \eqref{eq10302001}, we get
		$$\begin{aligned}
			0
			&\le\lim\limits_{k \to \infty}\inte|B(|\psi_k|)-B(|\varphi|)|dx\le C\lim\limits_{k \to \infty}\left[\left(\|\psi_k\|_{H^{\frac{1}{2}}(\mathbb{R}^N)}+\|\varphi\|_{H^{\frac{1}{2}}(\mathbb{R}^N)}\right)^{\frac{N}{N-1}}
			\|\psi_k-\varphi\|_{2}\right]=0.
		\end{aligned}$$
		Consequently,
		\begin{equation}\label{eq10302008}
			\lim\limits_{k \to \infty}\inte B(|\psi_k|)dx=\inte B(|\varphi|)dx.
		\end{equation}
		Moreover, following from  Fatou's Lemma, we obtain that
		\begin{equation}\label{eq10302015}
			\inte A(|\varphi|)dx \leq  \liminf_{k\to\infty}\inte A(|\psi_k|)dx.
		\end{equation}
		Thus, from \eqref{eq10302008} and \eqref{eq10302015}, it follows 
		\begin{equation}\label{eq10302027}
			\begin{aligned}
				\limsup_{k \to \infty}\inte|\psi_k|^2\log |\psi_k|^2dx
				&=\lim\limits_{k \to \infty}\inte B(|\psi_k|)dx-\liminf\limits_{k \to \infty}\inte A(|\psi_k|)dx\\
				&\le \inte B(|\varphi|)dx-\inte A(|\varphi|)dx=\inte|\varphi|^2\log |\varphi|^2dx.
			\end{aligned}
		\end{equation}
		Then, from \eqref{eq10302001}, \eqref{eq10302027} and weakly lower semi-continuity
		of the norm, we derive that
		\begin{equation}\label{eq10302059}
			\begin{aligned}
				d_{p}(a)
				&=\lim\limits_{k \to \infty} E_{p}(\psi_k)\\
				&\ge\frac{1}{2}\mathcal{T}_{m,v}(\varphi)-\frac{m}{2}\|\varphi\|_2^2-\frac{1}{p} \|\varphi\|_p^{p}-\frac{1}{2}\inte|\varphi|^2\log|\varphi|^2dx+\frac{1}{2}\|\varphi\|_2^2= E_{p}(\varphi)\ge d_{p}(a),
			\end{aligned}
		\end{equation}
		concluding the proof.
	\end{proof}

	\subsection{Existence and nonexistence of minimizers for the case \texorpdfstring{$p=2+\frac{2}{N}$}{}} 
	The main purpose of this subsection is to establish Theorem \ref{Thm01142024} on the existence and nonexistence of minimizers to problem (\ref{eq10271528}).
	We will divide it into two subsubsections.
	
	\subsubsection{Existence of minimizers for the case $p=2+\frac{2}{N}$}
	We start with the following proposition.
	\begin{prop}
		Assume that $m>0$, $a^\ast_v>a>0$ and $p=2+\frac{2}{N}$ with $N\ge 2$. 
		Then $E_{p}$ is bounded from below on $\mathcal{S}_a$.
	\end{prop}
	
	\begin{proof}
		Fix $u\in \mathcal{S}_a$ and $2<q<2+\frac{2}{N}$. By using \eqref{eq01161709}, \eqref{eq10272204}, \eqref{eq11280935} and  \eqref{eq102881138}, we infer that
		\begin{equation*}
			\begin{aligned}
				E_{p}(u )
				&\geq \frac{1}{2}T_v(u)-\frac{ma}{2}-\frac{1}{2}\left(\frac{a}{a^\ast_v}\right)^{\frac 1N}T_v(u)-\frac{1}{2}\inte [B(|u|)- A(|u|)]dx+\frac{a}{2}\\
				&\ge\frac{1}{2}\left[1-\left(\frac{a}{a^\ast_v}\right)^{\frac 1N}\right]T_v(u)-\frac{C_q}{2}\|u\|_q^q-\frac{ma}{2}\\
				&\ge\frac{1}{2}\left[1-\left(\frac{a}{a^\ast_v}\right)^{\frac 1N}\right]T_v(u)-\frac{C_q \widetilde{C}_q}{2}\left(T_v(u)\right)^{\frac{N(q-2)}{2} }a^{\frac{q}{2}-\frac{N(q-2)}{2} }-\frac{ma}{2}>C,
			\end{aligned}
		\end{equation*}
		for a suitable $C\in \R$, independent on $u$, concluding the proof.
	\end{proof}
	
	Next, arguing as in Lemma \ref{lem10281516} and Lemma \ref{lem10281714}, the following properties of $d_{p}(a)$ hold.
	\begin{lem}\label{lem01141059}
		Let $a^\ast_v>a>0$, $p=2+\frac{2}{N}$ with $N\ge 2$, and $m>m_p^*(a)$. Then  there holds
		\begin{equation*}
			d_{p}(a)< -\frac{1}{2}\left(1-\sqrt{1-|v|^2}\right )ma.
		\end{equation*}
	\end{lem}

	\begin{lem}\label{lem01141535}
		Under the hypotheses of Lemma \ref{lem01141059}, we have that
		\begin{equation*}
			d_{p}(a)<d_{p}(\lambda )+d_{p}(a-\lambda ), \text{ where } 0<\lambda<a. 
		\end{equation*} 
		Moreover, the functional $d_{p}(t)$ is strictly decreasing and continuous with respect to $t\in(0,a]$.
	\end{lem}
	
	\begin{proof}[Proof of Theorem \ref{Thm01142024}-\eqref{E}:]
		Repeating similar arguments as Theorem \ref{Thm10291145}, by using Lemma \ref{lem01141059} and Lemma \ref{lem01141535},  the desired conclusion follows. Since the proof is standard, here we omit it for simplicity.
	\end{proof}
	
	\subsubsection{Nonexistence of minimizers for the case $p=2+\frac{2}{N}$}
	Now, we prove the nonexistence of minimizers to problem (\ref{eq10271528}).
	
	\begin{proof}[Proof of Theorem \ref{Thm01142024}-\eqref{NE}]
		Let us define the 
		test function
		\begin{equation*}
			\phi_\tau:=\frac{\sqrt{a}\tau^{\frac{N}{2}}}{\|Q_v\|_2}Q_v(\tau \cdot), ~\tau>0,
		\end{equation*}
		where $Q_v\in H^{\frac{1}{2} }(\mathbb{R}^N)$ optimizes the inequality (\ref{eq11280935}) and satisfies equation (\ref{eq11280944}).

		Since $\phi_\tau \in \mathcal{S}_a$, by using \eqref{eq102881138} and \eqref{eq01071942}, direct calculations give that
		\begin{equation}\label{eq01071544}
			\begin{aligned}
				d_{p}(a)
				&\le E_{p}(\phi_\tau)
				\le\frac{1}{2}T_v(\phi_\tau)-\frac{N}{2N+2} \|\phi_\tau\|_{2+\frac{2}{N}}^{2+\frac{2}{N}}-\frac{1}{2}\inte|\phi_\tau|^2(\log|\phi_\tau|^2-1)dx \\
				&=\frac{\tau}{2}\frac{a}{a^\ast_v}T_v(Q_v)-\frac{N\tau}{2N+2}\left(\frac{a}{a^\ast_v}\right)^{\frac{N+1}{N}} \|Q_v\|_{2+\frac{2}{N}}^{2+\frac{2}{N}}-\frac{a}{2a^\ast_v}\inte|Q_v|^2\log\left(\frac{a\tau^N}{a^\ast_v}|Q_v|^2\right)dx+\frac{a}{2} \\
				&=\frac{N\tau a}{2}\left[1-\left(\frac{a}{a^\ast_v}\right)^{\frac{1}{N}}\right]-\frac{a}{2}\log\frac{a\tau^N}{a^\ast_v}
				-\frac{a}{2a^\ast_v}\inte|Q_v|^2\log|Q_v|^2dx+\frac{a}{2},
			\end{aligned}
		\end{equation}
		which, together with \eqref{eq10280849}, implies that, for any $a\ge a^\ast_v$, 
		\begin{equation*}
			\begin{aligned}
				d_{p}(a)
				&\le-\frac{a}{2}\log\frac{a}{a^\ast_v}-\frac{a}{2}\log\tau^N+C+\frac{a}{2}\to-\infty,~ as~ \tau\to +\infty.
			\end{aligned}
		\end{equation*}
		Hence, problem (\ref{eq10271528}) has no minimizers.
	\end{proof}
	
	\subsection{Nonexistence of minimizers for the case \texorpdfstring{$2+\frac{2}{N}<p\le \frac{2N}{N-1}$}{}} In this section, Theorem \ref{Thm1.3} is established by taking suitable test functions.
	
	\begin{proof}[Proof of Theorem \ref{Thm1.3}:]
		For any $u\in \mathcal{S}_a$, we let $u_t:=t^{\frac{N}{2}}u(t\cdot)$ with $t>0$. Then it follows from \eqref{eq102881138} that
		\begin{equation*}
			\begin{aligned}
				E_{p}(u_t)
				&\le \frac{1}{2}T_v(u_t)-\frac{1}{p} \|u_t\|_p^{p}-\frac{1}{2}\inte|u_t|^2(\log|u_t|^2-1)dx \\
				&= \frac{t}{2}T_v(u)-\frac{1}{p}t^{\frac{Np}{2}-N} \|u\|_p^{p}-\frac{Na}{2}\log t-\frac{1}{2}\inte|u|^2\log|u|^2dx +\frac{a}{2}\\
				&\to -\infty, \text{ as } t\to +\infty,
			\end{aligned}
		\end{equation*}
		where the fact that $\frac{Np}{2}-N>1$ is used. Thus, there exist no minimizers related to
		$d_p(a)$.
	\end{proof}
	
	\section{Blow-up analysis of minimizers}\label{sec5}
	In this section, taking in account Remark \ref{rem1.2}, we focus on showing the blow-up behaviour of the minimizers  as $a\nearrow a^\ast_v$ when $p=2+\frac{2}{N}$ and $m>0$. We start with the following estimate.
	
	\begin{lem}\label{lem01190906}
		Let $u_a$ be a minimizer corresponding to $d_p(a)$. Then, as $a\nearrow a^\ast_v$, we have that
		\begin{equation}\label{eq01191008}
			d_{p}(a)\to-\infty,
		\end{equation}
		\begin{equation}\label{eq01151218}
			\epsilon_ad_{p}(a)\to 0,   
		\end{equation}
		\begin{equation}\label{eq01151548}
			\epsilon_a\inte|u_a|^2\log|u_a|^2dx \to 0, 
		\end{equation}
		and
		\begin{equation}\label{eq01151545}
			\epsilon_a \|u_a\|_{2+\frac{2}{N}}^{2+\frac{2}{N}}\to \frac{N+1}{N}, 
		\end{equation}
		where
		\begin{equation}\label{eq01151059}
			\epsilon_a:=\left(T_v(u_a)\right)^{-1}\to 0^+.  
		\end{equation}
	\end{lem}
	\begin{proof}
		Recalling \eqref{eq10280849} and \eqref{eq01071544}, by taking $\tau:=\left[1-\left(\frac{a}{a^\ast_v}\right)^{\frac{1}{N}}\right]^{-\frac{1}{2}}$, we deduce that
		\begin{equation*}
			\begin{aligned}
				d_{p}(a)
				&\le
				\frac{N a}{2}\left[1-\left(\frac{a}{a^\ast_v}\right)^{\frac{1}{N}}\right]^{\frac{1}{2}}-\frac{a}{2}\log\frac{a}{a^\ast_v}-\frac{a}{2}N\log\left[1-\left(\frac{a}{a^\ast_v}\right)^{\frac{1}{N}}\right]^{-\frac{1}{2}}-\frac{a}{2a^\ast_v}\inte|Q_v|^2\log|Q_v|^2dx+\frac{a}{2}\\
				&\to -\infty, ~as~a\nearrow a^\ast_v.
			\end{aligned}
		\end{equation*}
		Moreover, from \eqref{eq10280938}, \eqref{eq11280935}, and \eqref{eq102881138}, it follows that
		\begin{equation*}
			\begin{aligned}
				d_{p}(a)
				&= E_{p}(u_a)
				\ge \frac{1}{2}\left[1-\left(\frac{a}{a^\ast_v}\right)^{\frac 1N}\right]T_v(u_a)-\frac{ma}{2}
				-\frac{1}{2}\inte|u_a|^2\log|u_a|^2dx+\frac{a}{2}\\
				&\ge-\frac{ma}{2}-\frac{1}{2}\inte|u_a|^2\log|u_a|^2dx,
			\end{aligned}
		\end{equation*}
		which, combined with \eqref{eq01191008}, implies that
		\begin{equation}\label{eq01151057}
			\inte|u_a|^2\log|u_a|^2dx\ge 2\left[-\frac{ma}{2}-d_{p}(a)\right]\to +\infty, \,as\,\, a\nearrow a^\ast_v.
		\end{equation}
		On the other hand, if we set 
		$\Omega_a:=\{x\in \R^N : |u_a(x)|\ge1\}$, then by \eqref{eq01151057}, for $a$ sufficiently close to $a_v^*$, we have $|\Omega_a|\neq0$. Thus
		for any $2<q<2+\frac{2}{N}$, we deduce from \eqref{eq10272204} that there exists a constant $C_q>0$, such that
		\begin{equation}\label{eq01151141}
			\begin{aligned}
				\inte|u_a|^2\log|u_a|^2dx
				&\le\int_{\Omega_a}|u_a|^2\log|u_a|^2dx\\
				&\le C_q\int_{\Omega_a}|u_a|^{q}dx \le C_q\|u_a\|_q^{q}\le C_q\widetilde{C}_q\left(T_v(u_a)\right)^{\frac{N(q-2)}{2} }a^{\frac{q}{2}-\frac{N(q-2)}{2}},
			\end{aligned}
		\end{equation}
		which, together with \eqref{eq01151057}, implies
		\begin{equation}\label{eq01151058}
			T_v(u_a)
			\to +\infty, \,as\,\, a\nearrow a^\ast_v
		\end{equation}
		and so we have \eqref{eq01151059}.
		
		Next, we prove \eqref{eq01151218}.
		Clearly, from \eqref{eq01191008} 
		and \eqref{eq01151058}, we obtain that 
		$$\limsup\limits_{a\nearrow a^\ast_v} \epsilon_ad_{p}(a)\le 0.$$
		To complete the proof it suffices to obtain $\liminf\limits_{a\nearrow a^\ast_v} \epsilon_ad_{p}(a)\ge 0$.
		From \eqref{eq11280935}, \eqref{eq102881138} and \eqref{eq01151059}, we infer that
		\begin{equation*}
			\begin{aligned}
				&\epsilon_a\left[\frac{1}{2}\int_{\mathbb{R}^N}\bar{u}_a(\sqrt{-\Delta+m^2 }-m+iv\cdot \nabla )  u_a dx
				-\frac{N}{2N+2} \|u_a\|_{2+\frac{2}{N}}^{2+\frac{2}{N}}\right]\\
				&\ge \epsilon_a\left[\frac{1}{2}\left(1-\left(\frac{a}{a^\ast_v}\right)^{\frac 1N}\right)T_v(u_a)-\frac{ma}{2}\right]=\frac{1}{2}\left[1-\left(\frac{a}{a^\ast_v}\right)^{\frac 1N}\right]-\epsilon_a\frac{ma}{2}\to 0, \,\,\,\text{as}\,\, a\nearrow a^\ast_v.
			\end{aligned}
		\end{equation*}
		Moreover, using \eqref{eq01151057}, \eqref{eq01151141}, \eqref{eq01151058} and $2<q<2+\frac{2}{N}$, we have that
		\begin{equation*}
			\begin{aligned}
				0<\frac{\epsilon_a}{2}\inte|u_a|^2\log|u_a|^2dx
				&\le\frac{\epsilon_a}{2} C_q\widetilde{C}_q\left(T_v(u_a)\right)^{\frac{N(q-2)}{2} }a^{\frac{q}{2}-\frac{N(q-2)}{2} }\\
				&=\frac{C_q\widetilde{C}_q}{2} \left(T_v(u_a)\right)^{\frac{N(q-2)}{2}-1 }a^{\frac{q}{2}-\frac{N(q-2)}{2} }\to 0, \,\,\,\text{as}\,\, a\nearrow a^\ast_v
			\end{aligned}
		\end{equation*}
		and so we can easily conclude obtaining \eqref{eq01151218} and \eqref{eq01151548}.
		
		Finally, we establish \eqref{eq01151545}.
		Combining \eqref{eq01151218} and \eqref{eq01151548}
		one can see that, as $a\nearrow a^\ast_v$,
		\begin{equation}\label{eq01151538}
			\begin{aligned}
				&\epsilon_a\left(\frac{1}{2}\int_{\mathbb{R}^N}\bar{u}_a(\sqrt{-\Delta+m^2 }-m +iv\cdot \nabla)  u_a dx-\frac{N}{2N+2} \|u_a\|_{2+\frac{2}{N}}^{2+\frac{2}{N}}\right)
				\to 0.
			\end{aligned}
		\end{equation}
		By using \eqref{eq102881138}, we deduce that
		\begin{equation*}
			\begin{aligned}
				T_v(u_a)-ma
				\le\int_{\mathbb{R}^N}\bar{u}_a(\sqrt{-\Delta+m^2 }-m+iv\cdot \nabla)  u_a dx
				\le T_v(u_a),
			\end{aligned}
		\end{equation*}
		which, combined with \eqref{eq01151059}, yields that
		\begin{equation}\label{eq01151123}
			\epsilon_a\int_{\mathbb{R}^N}\bar{u}_a(\sqrt{-\Delta+m^2 }-m +iv\cdot \nabla  )  u_a dx\to 1,
			\,as\,\, a\nearrow a^\ast_v.
		\end{equation}
		Then, from \eqref{eq01151538} and \eqref{eq01151123}, it  follows \eqref{eq01151545} and so the proof of Lemma \ref{lem01190906} is completed.
	\end{proof}

	\begin{rem}
		In fact, recalling the proof of the property \eqref{eq01191008},  we note that $\lim\limits_{a\nearrow a^\ast_v}d_{p}(a)=-\infty$ holds for any $m>0$.
	\end{rem}

	Based on the above precise estimates, we now prove Theorem \ref{Thm01152021}, namely, the blow-up behavior of the minimizers to problem (\ref{eq10271528}) as $a\nearrow a^\ast_v$. In the following, for simplicity, we write  $a\nearrow a^\ast_v$ instead of  $a_n\nearrow a^\ast_v$, eventually up to a subsequence.
	
	\begin{proof}[Proof of Theorem \ref{Thm01152021}]
		Let
		$$\tilde{\omega}_a:=\epsilon_a^{\frac{N}{2}}u_a(\epsilon_a\cdot),$$
		and observe that 
		\begin{equation}\label{L2n}
			\|\tilde{\omega}_a\|_2^2=\|u_a\|_2^2=a.
		\end{equation} 
		Moreover, from \eqref{eq01151545}, it follows  that
		\begin{equation}\label{eq01151612}
			\lim_{ a\nearrow a^\ast_v}\|\tilde{\omega}_a\|_{2+\frac{2}{N}}^{2+\frac{2}{N}}=\lim_{ a\nearrow a^\ast_v}\epsilon_a \|u_a\|_{2+\frac{2}{N}}^{2+\frac{2}{N}}= \frac{N+1}{N},
		\end{equation}
		and, using \eqref{eq01151059}, we further obtain
		\begin{equation*}
			T_v(\tilde{\omega}_a)=\epsilon_aT_v(u_a)=1,
		\end{equation*}
		which, together with \eqref{eq10280938}, implies that the sequence $\{\tilde{\omega}_a\}$ is bounded in $H^{\frac12}(\R^N)$. Then, by Lemma \ref{lem10301100}, \eqref{L2n}, and \eqref{eq01151612}, there exists a sequence $\{y_{a}\}\subset\mathbb{R}^N$ such that
		\begin{equation}\label{eq01151652}
			\lim\limits_{a\nearrow a^\ast_v}\int_{B_{1}(y_{a})}|\tilde{{\omega}}_{a} |^2dx>0.
		\end{equation}

		Set
		\begin{equation}\label{eq01151659}
			\omega_{a}:=\tilde{{\omega}}_{a} (\cdot+y_{a})=\epsilon _{a}^{\frac{N}{2} }u_{a}\left(\epsilon_{a}(\cdot+y_{a})\right).
		\end{equation}
		From (\ref{eq01151652}) we derive that
		\begin{equation}\label{eq01151660}
			\lim\limits_{a\nearrow a^\ast_v}\int_{B_{1}(0)}|{{\omega}}_{a} |^2dx>0.
		\end{equation}
		On one hand, since $\{\tilde{{\omega}}_{a}\}$ is bounded in $H^{\frac{1}{2}}(\mathbb{R}^N)$, then $\{\omega_{a}\}$ is bounded in $H^{\frac{1}{2}}(\mathbb{R}^N)$ and so, up to a subsequence, there exists $\omega_0\in H^{\frac{1}{2}}(\mathbb{R}^N)\setminus\{0\}$
		such that
		\begin{equation*}
			\omega_{a}\overset{}\rightharpoonup\omega_0 \text{~weakly~in} ~ H^{\frac{1}{2}}(\mathbb{R}^N)\,\, \text{as}\,\,a\nearrow a^\ast_v.
		\end{equation*}
		Then, from \eqref{eq01151660} and Fatou's lemma, we infer that
		\begin{equation}\label{eq01151728}
			\begin{aligned}
				0< \lim\limits_{a\nearrow a^\ast_v}\int_{B_1(0)}|{{\omega}}_{a} |^2dx=\int_{B_1(0)}|\omega_0|^2dx
				\leq\|{\omega}_0\|_2^{2}\leq\lim\limits_{a\nearrow a^\ast_v}\|{\omega}_{a}\|_2^{2}=a^\ast_v.
			\end{aligned}
		\end{equation}
		On the other hand, since $u_a$ is a minimizer of  problem \eqref{eq10271528}, it solves the following Euler-Lagrange equation
		\begin{equation}\label{eq01150823}
			(\sqrt{-\Delta+m^2 }-m+iv\cdot \nabla )u_a =\lambda_a u_a+ u_a\log|u_a|^2+|u_a|^{\frac2N}u_a,
		\end{equation}
		where the Lagrange multiplier $\lambda_a\in \R$ satisfies
		\begin{equation}\label{eq01150830}
			a\lambda_a=2d_{p}(a)- a
			-\frac{1}{N+1} \|u_a\|_{2+\frac{2}{N}}^{2+\frac{2}{N}}.
		\end{equation}
		By \eqref{eq01150830} and Lemma \ref{lem01190906}, we deduce that
		\begin{equation}\label{eq01151600}
			\epsilon_a\lambda_a\to -\frac{1}{ N a^\ast_v}, \,\,\text{as}\,\, a\nearrow a^\ast_v.
		\end{equation}
		In addition, using \eqref{eq01151659} and \eqref{eq01150823}, we notice that ${\omega}_{a}$ satisfies
		\begin{equation}\label{eq01151729}
			\left(\sqrt{-\Delta+(\epsilon_{a}m)^2 }-\epsilon_{a}m+ iv\cdot \nabla \right)  \omega_{a}=
			\epsilon_{a}\lambda_a\omega_{a}+\epsilon_{a}\omega_{a}\log|\epsilon_{a}^{-\frac N2}\omega_{a}|^2
			+|\omega_{a}|^{\frac{2}{N}}\omega_{a}.
		\end{equation}
		Since
		$$\inte \left|\omega_a \log|\omega_{a}|^2\varphi\right|dx
		\le C \inte \left(|\omega_a|^{\frac12}+|\omega_a|^2 \right)|\varphi|dx \le C, \,\, \varphi\in C_c^\infty(\R^N),$$ 
		by \eqref{eq01151059}, \eqref{eq01151600}, and the fact that  $\omega_{a}\overset{}\rightharpoonup\omega_0$
		weakly in $ H^{\frac{1}{2}}(\mathbb{R}^N)$, we get that $\omega_0$ satisfies
		\begin{equation}\label{eq01151730}
			(\sqrt{-\Delta} +iv\cdot \nabla )\omega_0= |\omega_0|^{\frac{2}{N} } \omega_0-\frac{1}{Na^\ast_v}\omega_0.
		\end{equation}
		Then the Pohozaev identity
		(see \cite[Lemma A.3]{BGLV19}) gives
		\begin{equation*}
			T_v(\omega_0)
			=\frac{1}{a^\ast_v}\|\omega_0\|_2^2=\frac{N}{N+1}\|\omega_0\|_{2+\frac2N}^{2+\frac2N}
		\end{equation*}
		and so, by (\ref{eq11280935}) and  (\ref{eq01151728}), it follows that
		\begin{equation*}
			\left({a^\ast_v}\right)^{\frac1N}
			=\frac{\|\omega_0\|_{2+\frac2N}^{2+\frac2N}}{\frac{N+1}{N\left(a^{\ast}_v\right)^{\frac1N}}T_v(\omega_0)}
			\leq\|\omega_0\|_2^{\frac2N}
			\leq\left({a^\ast_v}\right)^{\frac1N},
		\end{equation*}
		which indicates that $\omega_0$ optimizes the Gagliardo-Nirenberg inequality (\ref{eq11280935}) and $\omega_{a}$ converges to $\omega_0$
		strongly in $L^2({\mathbb{R}^N})$ as $a\nearrow a^\ast_v$. According to the interpolation inequality and Sobolev embedding theorem,
		we further have that
		$$\omega_{a}\overset{}\rightarrow\omega_0 \text{~in~} L^\gamma(\mathbb{R}^N) \text{ as } a\nearrow a^\ast_v , ~~\forall 2\le \gamma<2N/(N-1).$$
		Thus, by (\ref{eq01151729}), (\ref{eq01151730}) and \eqref{eq01151548}, we obtain that
		\begin{equation}\label{eq01151748}
			\lim_{a\nearrow a^\ast_v}\mathcal{T}_{\epsilon_{a}m}(\omega_{a})
			=T_v(\omega_0).
		\end{equation}
		Moreover, by \eqref{eq01151059}, we have that, as $a\nearrow a^\ast_v$,
		\begin{align*}
			0\le \int_{\mathbb{R}^N}\bar{\omega}_{a}\left(\sqrt{-\Delta+(\epsilon_{a}m)^2 }-\sqrt{-\Delta } \right)  \omega_{a}dx
			&=\int_{\mathbb{R}^N}\left(\sqrt{{|\xi |}^2+(\epsilon_{a}m)^2 }-{|\xi|}\right)|\mathcal{F}\left[{\omega}_{a}\right]|^2 d\xi\\
			&\leq\epsilon_{a}m\int_{\mathbb{R}^N}\left|\mathcal{F}\left[{\omega}_{a}\right ]\right|^2d\xi
			=\epsilon_{a}m\|{\omega}_{a}\|_2^2=\epsilon_{a}ma\to 0,
		\end{align*}
		which, together with \eqref{eq01151748}, yields that
		$$\lim_{a\nearrow a^\ast_v}T_v(\omega_a)=T_v(\omega_0).$$
		Hence,
		$$\omega_{a}\overset{}\rightarrow \omega_0 \text{~in} ~ H^{\frac{1}{2}}(\mathbb{R}^N), \text{ as }a\nearrow a^\ast_v.$$
		Let
		\begin{equation*}
			Q_0:={(Na^{\ast}_v)}^{\frac N2}\omega_0(Na^{\ast}_v\cdot).
		\end{equation*}
		Then $Q_0$ optimizes the Gagliardo-Nirenberg inequality (\ref{eq11280935}) and satisfies equation (\ref{eq11280944}).
		Furthermore, since $\|\omega_0\|_2^2=a^\ast_v$, we have that $\|Q_0\|_2^2=
		a^{\ast}_v$. Therefore,
		$$\epsilon _{a}^{\frac{N}{2} }u_{a}\left(\epsilon_{a}(\cdot+y_{a})\right)=\omega_{a}\overset{}\rightarrow \omega_0={(N\|Q_0\|_2^2)}^{-\frac N2}{Q_0\left(\frac{\cdot}{N\|Q_0\|_2^2}\right)}
		\text{~in~} H^{\frac{1}{2}}(\mathbb{R}^N), \text{ as }a\nearrow a^\ast_v,$$
		completing the proof.
	\end{proof}

	\section{Asymptotic behavior of minimizers}\label{sec6}
	This section aims to establish Theorem \ref{thm14.9}, that is, the limiting profiles of minimizers of problem (\ref{eq10271528}) for $a_n$ when $a_n$ tends to some $a_0$.

	\begin{proof}[Proof of Theorem \ref{thm14.9}]
		Here we only treat the case $2<p<2+\frac{2}{N}$ since the other one is similar.
		Let $\{a_n\}\subset(0,+\infty)$ such that $\lim\limits_{n\to \infty}a_n=a_0\in(0,+\infty)$ and, for each $n\in\mathbb{N}$, $u_{a_n}$ be a minimizer of problem \eqref{eq10271528} for $a_n$. Clearly, $u_{a_n}/\sqrt{a_n}$ is a minimizer  of the corresponding problem \eqref{eq10281724} for $a_n$.
		
		On one hand, by the definition of $d_p^{a_0}(1)$, for any $\epsilon>0$, there exists some $u_\epsilon\in W^{\frac{1}{2}}(\mathbb{R}^N)$ with $\|u_\epsilon\|_2^2=1$, such that
		\begin{equation*}
			E_p^{{a_0}}(u_\epsilon)\leq d_p^{a_0}(1)+\frac{\epsilon}{a_0}.
		\end{equation*}
		Then, by \eqref{eq01142202} and \eqref{eq10282013}, we have that
		\[
		\limsup_{n\to \infty}{d}_p (a_n)
		=\limsup_{n\to \infty}a_n{d}_p^{a_n} (1)
		\le a_0 \limsup_{n\to \infty} E_p^{a_n}(u_\epsilon)
		=a_0E_p^{a_0}(u_\epsilon)\leq a_0d_p^{a_0}(1)+\epsilon=d_p(a_0)+\epsilon,
		\]
		which, together with the arbitrariness of  $\epsilon$, indicates that
		\begin{equation}\label{eq19.9}
			\limsup_{n\to \infty}{d}_p (a_n)\leq d_p(a_0).
		\end{equation}
		
		Arguing as in the proof of Theorem \ref{Thm10291145}, we deduce that $\{u_{a_n}\}$ is bounded in $W^{\frac 12} (\R^N)$. Therefore, a simple calculation shows that
		\begin{equation}\label{eq14.18}
			\begin{aligned}
				\liminf_{n\to \infty}{d}_p (a_n)
				&=\liminf_{n\to \infty}a_n{d}_p^{a_n} (1)
				=\liminf_{n\to \infty}a_nE_p^{a_n}(a_n^{-\frac12}u_{a_n}) \\
				&=\liminf_{n\to \infty}a_n\left[E_p^{a_0}(a_n^{-\frac12}u_{a_n})+\frac{a_0^\frac{p-2}{2}-a_n^\frac{p-2}{2}}{pa_n^\frac{p}{2}} \|u_{a_n}\|_p^{p}+\frac{\log a_0 - \log a_n}{2}\right]\\
				&\geq a_0d_p^{a_0}(1)={d}_p (a_0).
			\end{aligned}
		\end{equation}
		As a result, from (\ref{eq19.9}) and \eqref{eq14.18} it follows that
		\begin{equation*}
			\lim_{n\to \infty}{d}_p(a_n)=\lim_{n\to \infty}E_p(u_{a_n})= d_p(a_0),
		\end{equation*}
		that is, $\{u_{a_n}\}$ is a minimizing sequence of $d_p(a_0)$.

		Repeating similar arguments as  the proof Theorem \ref{Thm10291145} (see (\ref{eq14.19})--(\ref{eq10302059})) and applying Lemma \ref{lem10261928}, there exists some ${u}_{a_0}\in W^{\frac{1}{2}}(\mathbb{R}^N)$ such that
		$u_{a_n}\rightarrow {u}_{a_0}$ in $W^{\frac{1}{2}}(\mathbb{R}^N)$ as $n\to \infty$, where ${u}_{a_0}$ is a minimizer of $d_p(a_0)$,  completing the proof.
	\end{proof}

	{\bf Acknowledgments} 
	P. d'Avenia and  A. Pomponio are members of GNAMPA-INdAM and were supported
	by GNAMPA-INdAM Project 2026 (CUP E53C25002010001) and by the Italian Ministry of University and Research
	under the Program Department of Excellence L. 232/2016 (CUP D93C23000100001).
	Q. H. He was supported by the fund from Natural
	Science Foundation of Guangxi (2025GXNSFFA069011), National Natural
	Science Foundation of China (Nos. 12061012,12461022) and Guangxi Bagui
	Young Top Talent Program.

	\smallskip
	
	{\bf Data availability} Data sharing not applicable to this article as no datasets were generated or analysed during the current study.
	
	\smallskip
	
	{\bf Conflict of interest} The authors declare that they have no conflict of interest.

\end{document}